\documentclass[11pt,reqno]{amsart}
\usepackage{diagrams}
\usepackage{caption}
\usepackage{array}
\usepackage{amssymb, mathtools}
\usepackage[pagewise]{lineno}
\numberwithin{equation}{section}
\usepackage{cite}
\usepackage{url}
\usepackage{graphicx}
\usepackage[all]{xy}
\usepackage{listings}
\makeatletter
\renewcommand*\env@matrix[1][*\c@MaxMatrixCols c]{%
	\hskip -\arraycolsep
	\let\@ifnextchar\new@ifnextchar
	\array{#1}}
\makeatother

\theoremstyle{definition}
\newtheorem{thm}{Theorem}[section]
\newtheorem{cor}[thm]{Corollary}
\newtheorem{lem}[thm]{Lemma}

\newtheorem{prop}[thm]{Proposition}
\newtheorem{defi}[thm]{Definition}
\newtheorem{rem}[thm]{Remark}

\newtheorem*{conj}{Conjecture}

\DeclareMathOperator{\im}{Im}

\DeclareMathOperator{\red}{\mathrm{red}}

\DeclareMathOperator{\Hc}{\mathcal{H}om}

\DeclareMathOperator{\Id}{\mathrm{Id}}

\DeclareMathOperator{\cn}{\mathbb{C}^{n}}
\DeclareMathOperator{\CC}{\mathbb{C}}

\DeclareMathOperator{\I}{\mathcal{I}}
\DeclareMathOperator{\mo}{\mathcal{O}}

\newcommand{\mr}[1]{\mathrm{#1}}
\newcommand{\mb}[1]{\mathbb{#1}}

\newcommand{\mc}[1]{\mathcal{#1}}
\newcommand{\ov}[1]{\overline{#1}}
\newcommand{\un}[1]{\underline{#1}}
\newcommand{\wt}[1]{\widetilde{#1}}
\newcommand{\mf}[1]{\mathfrak{#1}}
\begin{document}
	
	\title[Matrix factorization]{Matrix factorization for quasi-homogeneous singularities}
	
	\author{Ananyo Dan}

	\address{School of Mathematics and Statistics, University of Sheffield, Hicks building, Hounsfield Road, S3 7RH, UK}
	
	\email{a.dan@sheffield.ac.uk}

\author{Agust\'in Romano-Vel\'azquez}

\address{\parbox{\linewidth}{Alfréd Rényi Institute Of Mathematics, Hungarian Academy Of Sciences, Reáltanoda Utca 13-15, H-1053, Budapest, Hungary\newline Universidad Nacional Aut\'onoma de M\'exico  Avenida Universidad s/n, Colonia Lomas de Chamilpa CP 62210, Cuernavaca, Morelos Mexico}} 

\email{agustin@renyi.hu, agustin.romano@im.unam.mx}

	\thanks{}
	
	
	
	\date{\today}
	
	\begin{abstract}
		Given an isolated, quasi-homogeneous singularity $X$ we prove that 
		there is a group isomorphism 
		between the group of rank one reflexive sheaves on $X$ and the free abelian group generated by
		$\mb{C}^*$-divisors, modulo linear equivalence. When $\dim(X)=2$  we reduce the problem of finding matrix 
		factorizations of arbitrary reflexive $\mo_X$-modules to the same question on rank one reflexive 
		sheaves. We then enumerate the matrix factorizations of all rank one reflexive sheaves.
		As a consequence, we prove a conjecture of Etingof and Ginzburg on point modules.
	\end{abstract}

	\maketitle
	
	\section{Introduction}
	Let $X \subset \cn$ be an integral, normal hypersurface  defined by an equation 
	$F\in \CC[[X_1,\dots,X_n]]$. Recall, \emph{matrix factorizations} of $F$ are pairs of square matrices 
	$(M_1,M_2)$ of the same rank such that the products $M_1.M_2$ and $M_2.M_1$ equals $F$ times an identity matrix.
	Eisenbud~\cite{eish} showed that there is a one-to-one correspondence between (reduced) 
	matrix factorizations of $F$ and maximal Cohen-Macaulay $\mo_{X}$-modules without free direct summands. 
	Matrix factorization plays a central role
	in singularity theory. Using matrix factorization, Kn\"orrer~\cite{kno} and 
	Buchweitz-Greuel-Schreyer~\cite{buch} proved that isolated hypersurface singularities of finite Cohen-Macaulay representation type are exactly the simple 
	ones. In the early $2000$s, Kapustin~\cite{kapustin}, and Orlov~\cite{orlo1,orlo2,orlov2} showed that matrix factorizations can be applied to study Landau-Ginzburg models appearing in string theory, and to the study of Kontsevich's homological mirror symmetry. In particular, by the work of Orlov there exists an equivalence between 
	the bounded derived category $D^b(X)$ and the homotopy category of matrix factorizations of $F$. In general, the first category is hard to compute. Thus, producing concrete families of matrix factorizations can be one way of understanding $D^b(X)$.

	Unfortunately, there are no ``good'' algorithms to obtain matrix factorizations. As a result
	concrete examples of matrix factorizations are rather limited in the literature. 
	For example, Buchweitz, Eisenbud and Herzog~\cite{buch2} proved that for
	$F_n(X_1,\dots,X_n)=X_1^2 + \dots + X_n^n$ with $n \geq 8$ the smallest size of a matrix factorization is bounded below by $2^{\frac{n-2}{2}}\times 2^{\frac{n-2}{2}}$. In particular for $F_{16}$ the smallest matrix factorization is of size $128\times 128$. Crisler and Diveris~\cite{cris} produced an algorithm to produce matrix factorization for the polynomial $F_n$ only for $n \leq 8$. By studying the polynomial $F_{16}$ they notice that their algorithm fails and it is impossible to fix it.
	 Laza, Pfister and Popescu~\cite{lazam} computed all the matrix factorization associated to rank one
	 reflexive sheaves over the surface defined by $F_3$. 
	 Baciu~\cite{baci} computed all the matrix factorizations associated to rank two graded Ulrich modules on the
	hypersurface defined by $X_1^3 + X_1^2X_3 - X_2X_3$. Etingof and Ginzburg~\cite{ginz} produced a family of matrix factorizations for the family of hypersurfaces gives by the polynomial 
	$X_1^3 + X_2^3 + X_3^3 + \tau X_1X_2X_3$ as $\tau$ varies over non-zero complex numbers. 
	Ros Camacho and Newton~\cite{cama1,cama2} computed concrete matrix factorizations for exceptional 
	unimodal hypersurface singularities. The goal of this article is to generalize some of these results to 
	any isolated, quasi-homogeneous hypersurface singularity (upto topologically trivial deformations).

	Let $(X,x)$ be an isolated, quasi-homogeneous hypersurface singularity of dimension $2$.
	This means that there exist integers $(\omega_1, \omega_2,\omega_3,d)$ such that the defining equation $F$
	satisfies: \[F(\lambda^{\omega_1}X_1,\lambda^{\omega_2}X_2,\lambda^{\omega_3}X_3)=\lambda^dF(X_1,X_2,X_3),\, 
	\mbox{ for all } \lambda \in \mb{C}^*.\]
	The integers $\omega_1, \omega_2,\omega_3$ are called the \emph{weights} of the hypersurface.
	Note first that every maximal Cohen-Macaulay module $M$ on $X$ 
	sits in an exact sequence with $4$ terms.
	Besides $M$ the remaining three terms are a trivial bundle,  a skyscraper sheaf 
	supported on the singular point $x$ and a rank one reflexive sheaf $\mc{L}$, which we will call the 
	\emph{determinant} of 
	$M$ (Theorem \ref{thm:mck}). Projective resolutions of skyscraper sheaves 
	are well-understood. Moreover, to obtain projective resolutions of short exact sequences, one simply needs to 
	determine the projective resolution of two of the three terms (satisfying the obvious compatibility conditions).
	As a result, finding the matrix factorization corresponding to $M$ reduces to determining the matrix factorization 
	corresponding to its determinant $\mc{L}$. We first classify all such rank one reflexive sheaves. 
	Denote by $\mr{Ref}^{(1)}(X)$ the group of all reflexive rank one sheaves on $X$ (see \S \ref{sec:rank1}
	for the group structure) and by $\mc{D}(X)$ the free abelian group generated by classes of $\mb{C}^*$-curves
	(i.e., curves that are invariant under the natural $\mb{C}^*$-action on $X$, see \S \ref{sec:div}), 	modulo linear equivalence. 
		We prove:
	
	\begin{thm}\label{thm}
		Any integral curve $D$ in $X$ is either a $\mb{C}^*$-curve or is CI-linked (see Definition \ref{defi:ci})
		to a $\mb{C}^*$-curve. Moreover, there is an isomorphism of abelian groups:
		\begin{equation}\label{eq03}
		 \mc{D}(X) \to \mr{Ref}^{(1)}(X)\, \mbox{ sending } D \in \mc{D}(X) \mbox{ to } i_*\mo_{X^*}(D \cap X^*),
		\end{equation}
		where $X^*:=X\backslash \{x\}$ is the regular locus in $X$ and $i:X^* \to X$ is the open immersion.
	\end{thm}
	See Theorems \ref{thm:ci} and \ref{thm:ref} for a more general statement that holds in any dimension of $X$.
	This can be viewed as a McKay-type correspondence where the left hand side of the correspondence 
	\eqref{eq03} parameterizes geometric objects namely $\mb{C}^*$-divisors and the right hand side parameterizes
	algebraic objects namely rank one reflexive sheaves.
	
	In arbitrary rank, there is a $1-1$ correspondence between maximal Cohen-Macaulay $\mo_X$-modules 
	and rank one Cohen-Macaulay $\mo_X$-modules supported on divisors (see Proposition \ref{prop:corr}).
	This correspondence associates to a rank $r$ maximal Cohen-Macaulay $\mo_X$-modules $M$ along with 
	a general choice of $r$ sections, its degeneracy module. The advantage of this correspondence is that 
	one can obtain the matrix factorization of $M$ from a projective resolution of the associated
	degeneracy module (Theorem \ref{thm2}). The latter is an easier problem. We use this idea in the proof of Theorem \ref{thm:Quasi} below.
	
	By Theorem \ref{thm} above, rank one maximal Cohen-Macaulay modules are generated (via tensor product) by 
	those arising from integral $\mb{C}^*$-curves. As a result, maximal Cohen-Macaulay modules associated to 
	non-singular $\mb{C}^*$-curves are of particular interest. We call such modules \emph{generalized Wunram modules} (see \S \ref{sec:gen-wun}). 
	 We give an explicit description of the matrix factorization corresponding to rank one 
    generalized Wunram modules in Theorem \ref{thm:Quasi} below. 
    Note that, $X$ contains a non-singular $\mb{C}^*$-curve if and only if 
    (upto reparametrization) one of the 
	weights of $X$ is one. Recall, 
	Orlik and Wagreich~\cite{orli} and Arnold~\cite{arn1} classified isolated quasi-homogeneous
	surface singularities, upto topologically trivial deformations (see table in \S \ref{sec:type}). 
    Corresponding to the types of singularities mentioned in this table we derive the following list of 
    matrix factorizations: 
	
	\begin{thm}\label{thm:Quasi}
	 Let $X$ be a quasi-homogeneous singularity of weight $(1, \omega_2, \omega_3)$ listed in Table \ref{tabl1} in \S \ref{sec:type} below. 
	 Given positive integers $n, m$ and complex numbers $c_1, c_2$, denote by:
	\[				S_{(c_1,c_2,n,m)}(Z_1,Z_2):=\sum_{j=1}^m \frac{Z_1^{(j-1)n}Z_2^{m-j}c_1^{j-1} }{c_2^{jn}}\]
		Then,  the matrix factorization associated to any rank one generalized Wunram module on $X$ is a pair of 
		$2 \times 2$ matrices $(\mr{adj}(A),A)$ where $A$ is a matrix of the form $A:=(m_{i,j})$ 
		for $i,j \in \{1,2\}$ with 
		\[m_{1,1}= X_1^{\omega_2}b-X_2a^{\omega_2},\, m_{1,2}= X_3a^{\omega_3}-X_1^{\omega_3}c\, \mbox{ and }\]
		the entries $m_{2,1}, m_{2,2}$ are given by the following table where the first column enumerates the various singularity types from Table \ref{tabl1}:
		{\small \begin{center}
\begin{tabular}{ | m{1.7cm} | m{7.4cm} | m{4.6cm} | } 
\hline
 Type &  $m_{2,1}$ & $m_{2,2}$\\
  \hline
  $\mr{I}_{p,q,r}$   & $S_{(c,a,\omega_3,r)}(X_1,X_3)$ & $S_{(b,a,\omega_2,q)}(X_1,X_2)$ \\ 
  \hline
   $\mr{II}_{p,q,r}$   & \vspace{1mm} $\frac{bX_1^{\omega_2}}{a^{\omega_2}}S_{(c,a,\omega_3,r)}(X_1,X_3)$ & $S_{(b,a,\omega_2,q)}(X_1,X_2)+$ $+X_3^rS_{(b,a,\omega_2,1)}(X_1,X_2)$ \\ 
  \hline 
  $\mr{III}_{p,q,r}$   & \vspace{1mm} $\frac{bX_1^{\omega_2}}{a^{\omega_2}}S_{(c,a,\omega_3,r)}(X_1,X_3)$ $+ \frac{b^qX_1^{q\omega_2}}{a^{q\omega_2}}S_{(c,a,\omega_3,1)}(X_1,X_3)$ & $X_3S_{(b,a,\omega_2,q)}(X_1,X_2)+$ $+X_3^rS_{(b,a,\omega_2,1)}(X_1,X_2)$  \\ 
  \hline 
  $\mr{IV}_{p,q,r}$   & \vspace{1mm} $XS_{(c,a,\omega_3,r)}(X_1,X_3)$ $+ \frac{b^qX_1^{q\omega_2}}{a^{q\omega_2}}S_{(c,a,\omega_3,1)}(X_1,X_3)$ & $X_3S_{(b,a,\omega_2,q)}(X_1,X_2)$ \\ 
  \hline 
  $\mr{V}_{p,q,r}$   & $XS_{(c,a,\omega_3,r)}(X_1,X_3)$ $+ \frac{b^qX_1^{q\omega_2}}{a^{q\omega_2}}S_{(c,a,\omega_3,1)}(X_1,X_3)$ & $X_3S_{(b,a,\omega_2,q)}(X_1,X_2)+$ $+X_1^pS_{(b,a,\omega_2,1)}(X_1,X_2)$ \\ 
  \hline 
  $\mr{VI}_{p,q,r,b_2,b_3}$   & $XS_{(c,a,\omega_3,r)}(X_1,X_3)$ $+ \frac{b^{b_2}X_1^{b_2\omega_2}}{a^{b_2\omega_2}}S_{(c,a,\omega_3,b_3)}(X_1,X_3)$ &  $XS_{(b,a,\omega_2,q)}(X_1,X_2)+$
  $+X_3^{b_3}S_{(b,a,\omega_2,b_2)}(X_1,X_2)$\\ 
  \hline 
  $\mr{VII}_{p,q,r,b_2,b_3}$   & $XS_{(c,a,\omega_3,r)}(X_1,X_3)$ $+ \frac{b^{b_2}X_1^{b_2\omega_2}}{a^{b_2\omega_2}}S_{(c,a,\omega_3,b_3)}(X_1,X_3)$ & $XS_{(b,a,\omega_2,q)}(X_1,X_2)+$ $+X_3^{b_3}S_{(b,a,\omega_2,b_2)}(X_1,X_2)+$ $+X_1^pS_{(b,a,\omega_2,1)}(X_1,X_2)$\\ 
  \hline
\end{tabular}
\end{center}}
where $(a,b,c)$ varies over all points in $X$ with $a \not= 0$.
	\end{thm}

	This result will be proved in \S \ref{sec:proof}. The remaining case $a=0$ is treated in Remark \ref{rem:comp}.
	Note that this also gives explicit families of matrix factorizations parameterized by points on $X$.
	Our computation recovers the matrix factorizations obtained by Laza, Pfister and Popescu~\cite{lazam}.  As a consequence of Theorem \ref{thm:Quasi} above we prove special cases of a conjecture of Etingof and Ginzburg \cite[Conjecture $3.6.8$]{ginz}:
	
	\begin{conj}
	 Let $F$ be the free tensor algebra with basis $X_1, X_2, X_3$, $\Phi \in F/[F,F]$, $\mf{A}(\Phi):= F/ \langle \langle \partial_i \Phi\rangle \rangle_{i=1,2,3}$ and for a central element $\Psi$ not a zero divisor in $\mf{A}(\Phi)$ denote by  $\mf{B}(\Phi,\Psi):= \mf{A}(\Phi)/ \langle \langle \Psi\rangle \rangle$. To any point module $P$ (see \cite[Definition $3.8$]{art3}) over the algebra $\mf{B}(\Phi,\Psi)$
	 one can naturally associate a matrix factorization $M(P)=(M_+,M_-)$.
	\end{conj}

	Using Theorem \ref{thm:Quasi} we prove:
	
	\begin{thm}
	 Take $\Phi:= X_1X_2X_3-X_2X_1X_3$. Then, for suitable choices of $\Psi$ the above conjecture holds true
	 i.e., to any point module $P$  over the algebra $\mf{B}(\Phi,\Psi)$
	 one can naturally associate a matrix factorization.
	\end{thm}

	See Theorem \ref{thm:conj} for the precise statement. Note that, the choices of $\Psi$ in the above theorem 
	will correspond to quasi-homogeneous polynomials.

	In Section~\ref{sec:cusp} we study the case of cusp singularities.
	In the workshop of Singularities at Oberwolfach 2021, Prof.  Duco van Straten asked a question to
	the second author on the construction of matrix factorizations for cusp singularities. 
	We obtain a partial answer to his question, in Theorem~\ref{th:Cusp}. In particular,
	we produce families of matrix factorizations for families of cusp singularities. By fixing some numbers, this theorem also recovers the cubic studied by Etingof and Ginzburg~\cite{ginz}.  In Section~\ref{sec:nonIsolated} we study matrix factorization of non-isolated singularities
	and generalize a result of Baciu~\cite{baci}.

	\section{Preliminaries}
	In this section, we recall the notion of matrix factorization of hypersurface singularities.
	We observe how this relates to the space of maximal Cohen-Macaulay modules (Theorem \ref{Th:EisTh}).
	Finally, we recall basics on degeneracy modules (Proposition \ref{prop:corr}). 
	This gives us a new approach to studying matrix factorizations, which will be used in later sections for 
	explicit computations.
	
	\subsection{Setup}\label{se01}
	Fix an integer $n \ge 3$.
	Let $X$ be an integral, normal hypersurface in $\cn$.  Denote by $\mo_{\cn}:= \mb{C}[[X_1,...,X_n]]$ and $F \in \mo_{\cn}$
	defining the hypersurface $X$ and $\mo_X:=\mo_{\cn}/(F)$ the associated coordinate ring. 
	Note that $X$ may have \emph{non-isolated} singularities.
	
	\subsection{Matrix factorization}
			 A \emph{matrix factorization} of $F$ is an ordered pair of $m \times m$-matrices $(\Phi, \Psi)$ with entries in $\mo_{\cn}$ such that the matrix multiplication satisfies:
		\begin{equation*}
			\Phi \cdot \Psi = F \cdot \mathrm{Id}_m, \quad \Psi \cdot \Phi = F \cdot \mathrm{Id}_m,
		\end{equation*}
		where $\mathrm{Id}_m$ is the $m\times m$ identity matrix.
				The matrix factorization is \emph{reduced} if and only if
		\begin{equation*}
			\im (\Phi: \mo_{\cn}^{\oplus m} \to \mo_{\cn}^{\oplus m}) \subset \mathfrak{m} \mo_{\cn}^{\oplus m} \quad \text{and} \quad \im(\Psi: \mo_{\cn}^{\oplus m} \to \mo_{\cn}^{\oplus m}) \subset \mathfrak{m} \mo_{\cn}^{\oplus m},
		\end{equation*}
	where $\mf{m}$ is the maximal ideal of $\mo_{\cn}$. Recall, the following classical result on matrix factorization:
	
	\begin{thm}
		\label{Th:EisTh}
		There is a one-to-one correspondence between:
		\begin{enumerate}
			\item equivalence classes of reduced matrix factorizations of $F$.
			\item isomorphism classes of non-trivial periodic minimal free resolutions of $\mo_X$-modules of periodicity two.
			\item maximal Cohen-Macaulay $\mo_X$-modules without free summands.
		\end{enumerate}
	\end{thm}
	
	\begin{proof}
	 See \cite[Corollary $6.3$]{eish} for a proof.
	\end{proof}
	
	In this article, we will exploit the equivalence between $(1)$ and $(3)$
	in Theorem \ref{Th:EisTh}. 
	So, we briefly recall how one associates a matrix factorization
	of $F$ to a maximal Cohen-Macaulay module without free summands.
	Let $M$ be a maximal Cohen-Macaulay $\mo_X$-module without free summands. This implies that the depth of $M$ equals the dimension of $X$, which is $n-1$.
	By the Auslander-Buchsbaum formula, this means as an $\mo_{\cn}$-module, the projective dimension of $M$ equals $1$. This implies we have a short exact sequence of the form
	\begin{equation}\label{eq:std}
		0 \to \mo_{\cn}^{\oplus b} \xrightarrow{\Phi} \mo_{\cn}^{\oplus a} 
		\xrightarrow{(m_1,...,m_a)} M \to 0
	\end{equation}
	where $m_i \in M$ and the standard basis element $e_i \in \mo_{\cn}^{\oplus a}$ maps to $m_i$ for $1 \le i \le a$.
	Since $M$ is supported on $X$, we have $a=b$. Then the morphism $\Phi$ is simply given by an $a \times a$-matrix with entries in $\mo_{\cn}$. Suppose that this is a minimal resolution of $M$. Since $M$ is annihilated by $F$, 
    for every $1 \le i \le a$, $Fe_i \in \mo_{\cn}^{\oplus a}$ maps to
	zero in $M$. By the exactness of \eqref{eq:std}, 
	there exists $\Psi(e_i) \in \mo_{\cn}^{\oplus a}$ such that 
	$\Phi \circ \Psi(e_i)=Fe_i$. In other words, there exists an $a \times a$-matrix $\Psi$ with entries in $\mo_{\cn}$ such that $\Phi \cdot \Psi = F \cdot \mathrm{Id}_a$. 
		Therefore, $(\Psi,\Phi)$ is a matrix factorization of $F$.

	\subsection{Degeneracy module}
	Let $M$ be a maximal Cohen-Macaulay $\mo_X$-module of rank, say $r$.
	Given an $r$-tuple of sections $\un{s}:=(s_1,...,s_r)$ of $M$, the 
	associated \emph{degeneracy locus} is the zero locus of the section 
	$s_1 \wedge s_2 \wedge ... \wedge s_r \in \wedge^r M$ i.e., the locus 
	of points where the $r$-tuple of sections is linearly dependant. 
	Consider the morphism \[\un{s}: \mo_X^{\oplus r} \to M,\]
	sending a standard basis vector $e_i$ of $\mo_X^{\oplus r}$ to $s_i$.
	Denote by $\mc{C}_{\un{s}}$ the cokernel of the morphism $\un{s}$.
	Note that, the support of $\mc{C}_{\un{s}}$ is the associated degeneracy locus.
	For a 
	general choice of $r$-sections $\un{s}:= (s_1,...,s_r)$ the associated degeneracy
	locus $\mr{Supp}(\mc{C}_{\un{s}})$ is reduced and Cohen-Macaulay of 
	codimension $1$ (see \cite[p. $431$]{eis32}).
	Furthermore, by the genericity of the $r$-tuple of sections, 
	the locus where $r-1$ of the $r$-sections are linearly dependant is of codimension
	$2$ (see \cite[Lemma $5.2$]{eis32}). This implies that the cokernel $\mc{C}_{\un{s}}$
	is supported on a reduced Cohen-Macaulay subvariety of codimension $1$ in $X$   
	and is of rank $1$ over its support. 
	The cokernel $\mc{C}_{\un{s}}$ will be called the \emph{degeneracy module associated to the 
		$r$-tuple of sections} $\un{s}:= (s_1,...,s_r)$.
	This motivates the following definition:
	
	\begin{defi}
		We will call an $r$-tuple of sections $\un{s}:=(s_1,...,s_r)$ of $M$ 
		\emph{weakly general} if the cokernel $\mc{C}_{\un{s}}$ of 
		the induced morphism $\un{s}$ is supported on a reduced Cohen-Macaulay subvariety of $X$ 
		of codimension $1$  
		and is a rank $1$, Cohen-Macaulay $\mo_X$-module over $\mr{Supp}(\mc{C}_{\un{s}})$.
	\end{defi}
	
	\subsection{Dualizing degeneracy modules}
	Let $M$ be a maximal Cohen-Macaulay $\mo_X$-module of rank $r$ and 
	$\un{s}:=(s_1,...,s_r)$ be an $r$-tuple of weakly general sections of $M$.
	By definition, we have a short exact sequence of the form:
	\begin{equation}\label{eq06}
		0 \to \mo_X^{\oplus r} \xrightarrow{\un{s}} M \to \mc{C}_{\un{s}} \to 0,
	\end{equation}
	for some Cohen-Macaulay $\mo_X$-module $\mc{C}_{\un{s}}$ supported on a reduced
	Cohen-Macaulay subvariety in $X$ and is of rank $1$ over its support.
	Dualizing this exact sequence, we get 
	\begin{equation}\label{eq05}
		0 \to M^{\vee} \to \mo_X^{\oplus r} \xrightarrow{\un{s}'}
		\mc{E}xt^1_X(\mc{C}_{\un{s}},\mo_X) \to 0,
	\end{equation}
	where the surjectivity on the right follows from the vanishing of 
	$\mc{E}xt^1_X(M,\mo_X)$ (see \cite[Theorem $3.3.10$]{brun1}).
	Throughout this article, we shall denote $\mc{A}_{\un{s}}:= \mc{E}xt^1_X(\mc{C}_{\un{s}},\mo_X)$.
	Note that, dualizing \eqref{eq05} and using $\mc{E}xt^1_X(\mc{A}_{\un{s}},\mo_X) \cong \mc{C}_{\un{s}}$
	and $M^{\vee \vee} \cong M$ (see \cite[Theorem $3.3.10$]{brun1}), we get back the exact 
	sequence \eqref{eq06}. This implies:
	\begin{prop}\label{prop:corr}
		There is a $1-1$ correspondence between pairs:
		\[\begin{Bmatrix}
			& (M,\un{s}) \mbox{ where } M \mbox{ is a MCM}\\
			& \mo_X\mbox{-module of rank } r \mbox{ and }\\
			& \un{s}:=(s_1,...,s_r) \mbox{ is an } r\mbox{-tuple}\\
			& \mbox{ of weakly general sections of } M
		\end{Bmatrix}\longleftrightarrow
		\begin{Bmatrix}
			& (\mc{A}_{\un{s}}, \un{s}') \mbox{ where } \mc{A}_{\un{s}} \mbox{ is a CM }\mo_X\mbox{-module}\\
			& \mbox{supported on a CM subvariety of}\\ 
			& \mbox{codimension one in } X \mbox{ and of rank } 1\\
			& \mbox{ over the support } \mbox{and generated by } \un{s}'
		\end{Bmatrix}\]
		where the bijection follows from \eqref{eq06} and \eqref{eq05}.
	\end{prop}
	
	\begin{proof}
	 See \cite{BoRo} for a detailed proof.
	\end{proof}

	\begin{defi}
		Given a pair $(M,\un{s})$ with $M$ a maximal Cohen-Macaulay
		$\mo_X$-module of rank $r$ and $\un{s}$ an $r$-tuple of weakly general sections of $M$, we will 
		call the corresponding pair $(\mc{A}_{\un{s}},\un{s}')$ as in Proposition \ref{prop:corr},
		the \emph{degenerate pair associated to} $(M,\un{s})$.
	\end{defi}

	\section{McKay-type correspondence for quasi-homogeneous singularities}
	Quasi-homogeneous hypersurface singularities are generalizations of homogeneous singularities. 
	We study $\mb{C}^*$-divisors contained in such hypersurfaces. We observe that every effective, integral divisor is 
	either a $\mb{C}^*$-divisor or is CI-linked (in the sense of Definition \ref{defi:ci}) 
	to a $\mb{C}^*$-divisor (Theorem \ref{thm:ci}).
	Using this we observe that there is a $1-1$ correspondence between $\mb{C}^*$-divisors (modulo linear equivalence)
	and rank one 
	reflexive sheaves
	on a quasi-homogeneous hypersurface (Theorem \ref{thm:ref}). Furthermore, if the dimension 
	of the hypersurface is two, then we can express every maximal Cohen-Macaulay modules solely in terms 
	of the ideal sheaves of $\mb{C}^*$-curves and certain skyscraper sheaves (Theorem \ref{thm:mck}).
	
	\subsection{Quasi-homogeneous hypersurfaces}\label{sec:quasi}
	A polynomial $F \in \mb{C}[[X_1,X_2,...,X_n]]$ is called \emph{quasi-homogeneous}
	if there exists positive integers $(\omega_1, \omega_2,...,\omega_n,d)$ such that 
	for any $\lambda \in \mb{C}^*$, we have $F(\lambda X_1, \lambda X_2,...,\lambda X_n)=\lambda^d F(X_1,....,X_n)$.
	The hypersurface $X$ defined by $F$ is called a \emph{quasi-homogeneous hypersurface with weights} 
	$\un{\omega}:= (\omega_1, \omega_2,...,\omega_n)$. Note that, there is a natural $\mb{C}^*$-action on $X$:
	\[\mb{C}^* \times X \to X \mbox{ sending } (\lambda, (x_1,...,x_n)) \mapsto (\lambda^{\omega_1} x_1, 
	 \lambda^{\omega_2} x_2,...,\lambda^{\omega_n} x_n).	\]
    Throughout this section we assume that $X$ has only isolated singularity at the origin $0$.
    Denote by $\mb{P}^{\un{\omega}}_{X^*}$ the quotient of $X^*:=X \backslash \{0\}$
    by the $\mb{C}^*$-action. Consider the resulting quotient map:
    \begin{equation}\label{eq:quot}
     \pi_X: X^* \to \mb{P}^{\un{\omega}}_{X^*}.
    \end{equation}

    \subsection{$\mb{C}^*$-divisors}\label{sec:div}
    Let $X$ be a quasi-homogeneous hypersurface of dimension $n$
    with weights $\un{\omega}:= (\omega_1, \omega_2,...,\omega_n,d)$.
    Note that,
    given a closed point $(a_1,a_2,...,a_n) \in X$, the associated $\mb{C}^*$-curve is the parametric curve 
    given by:
    \[n: \mb{C}^* \to X \mbox{ sending } \lambda \mapsto (\lambda^{\omega_1} a_1, 
	 \lambda^{\omega_2} a_2,... \lambda^{\omega_n} a_n).	\]
    We will denote by $[a_1,a_2,...,a_n]$ the corresponding point on $\mb{P}^{\un{\omega}}_{X^*}$. 
    Clearly, the fiber over $[a_1,...,a_n]$ to the morphism $\pi_X$
    is an integral curve and $n$ is the normalization map for the fiber.
    This implies that the preimage under $\pi_X$ of an integral divisor in $\mb{P}^{\un{\omega}}_{X^*}$ is irreducible.
        
    An integral divisor $D$ in $X^*$ is called a $\mb{C}^*$-\emph{divisor} if there exists an integral Weil divisor 
    $D'$ in $\mb{P}^{\un{\omega}}_{X^*}$ such that $D \cong \pi_X^{-1}(D')_{\red}$, where 
    $\pi_X$ is as in \eqref{eq:quot}. An integral divisor in $X$ is called $\mb{C}^*$-\emph{divisor} if 
    it is the closure of an integral $\mb{C}^*$-divisor on $X^*$. Denote by $\mc{D}(X)$ the free abelian group generated by 
    integral $\mb{C}^*$-divisors in $X$, modulo linear equivalence. Elements of $\mc{D}(X)$ will be called $\mb{C}^*$-\emph{divisors} on $X$.

	\subsection{Liaisons and residual divisors}
	Let $(X,0)$ be an isolated, quasi-homogeneous hypersurface singularity with weights 
	$\un{\omega}:= (\omega_1, \omega_2,...., \omega_n, d)$.
	Consider the quotient map $\pi_X$ as in \eqref{eq:quot} from the regular locus of $X$ to quotient by the 
	$\mb{C}^*$-action. 
	
	\begin{defi}
	 An integral divisor $D \subset X$ is called \emph{horizontal} if the composition 
	 \[D\backslash \{0\} \subset X \backslash \{0\} \xrightarrow{\pi_X} \mb{P}^{\un{\omega}}_{X^*}\]
	 is dominant.
	\end{defi}

	\begin{defi}\label{defi:ci}
	 Two distinct divisors $D, E$ are called \emph{CI-linked} if there exists a 
	 polynomial $g \in \mb{C}[[X_1,X_2,...,X_n]]$
	 such that $D \cup E = Z(g) \cap X$, 
	 where $Z(g)$ denotes the zero locus of $g$.
	 Moreover, if $D$ and $E$ are CI-linked then we call $D$ \emph{residual to} $E$ (and vice versa, $E$ is 
	 residual to $D$). This terminology is inspired by the classical theory of liaisons (see \cite{pesk}).
	\end{defi}

	\begin{thm}\label{thm:ci}
	 Let $D \subset X$ be an integral
	 horizontal divisor. Then, there exist a $\mb{C}^*$-divisor $E \subset X$
	 such that $D$ is CI-linked to $E$.
	\end{thm}

	\begin{proof}
	 Consider the quotient map $\pi_X$ from $X^*$ to $\mb{P}^{\un{\omega}}_{X^*}$ as in 
	 \eqref{eq:quot}. By the theorem on generic smoothness, there exists an open dense 
	 affine subscheme $U \subset \mb{P}^{\un{\omega}}_{X^*}$ such that the resulting morphism 
	 from $\pi_X^{-1}(U)$ to $U$ is smooth. Since $\pi_X$ is an affine morphism and 
	 $U$ is affine, we have $\pi_X^{-1}(U)$ is affine and non-singular. This implies 
	 $\mr{Pic}(\pi_X^{-1}(U))=0$. As $D$ is an integral horizontal divisor, $U_D:= D \cap \pi_X^{-1}(D)$
	 is a non-empty Cartier divisor in $\pi_X^{-1}(U)$. Since  $\mr{Pic}(\pi_X^{-1}(U))=0$, the ideal sheaf
	 $\I_{U_D}$ is simply $f.\mo_{\pi^{-1}_X(U)}$ for some $f \in \mo_{\pi^{-1}_X(U)}$.
	 By \cite[Lemma II.$5.3$]{R1}, there exists a regular function $\wt{f} \in \mo_X$ such that 
	 $Z(\wt{f}) \cap \pi_X^{-1}(U)=Z(f) \cap \pi_X^{-1}(U)$. This implies that the zero locus $Z(\wt{f})$ of 
	 $\wt{f}$ is of the form 
	 \begin{equation}\label{eq:ci}
	  Z(\wt{f}) = \ov{Z(f)} \cup E
	 \end{equation}
	 where $\ov{Z(f)}$ is the closure in $X$ of the zero locus of $f$ and $E$ is a divisor lying 
	 in the complement $X \backslash \pi_X^{-1}(U)$. Since $E$ is a divisor and does not intersect $\pi_X^{-1}(U)$,
	 the scheme-theoretic image $\pi_X(E)$ of $E$ in $\mb{P}^{\un{\omega}}_{X^*}$ does not intersect $U$. 
	 Since the fibers of $\pi_X$ are irreducible and of dimension one, we conclude by the 
	 fiber dimension theorem that $E \cong \pi_X^{-1}(E')$ for some divisor $E'$ in $\mb{P}^{\un{\omega}}_{X^*}$.
	 In particular, $E$ is a $\mb{C}^*$-divisor. 
	 Moreover, as $D$ is integral and agrees with $Z(f)$ over $\pi_X^{-1}(U)$, we have 
	 $\ov{Z(f)} = D$. By \eqref{eq:ci}, this means $D$ is CI-linked to a $\mb{C}^*$-divisor $E$. 
	 This proves the theorem.
	\end{proof}

   \subsection{Rank one correspondence}\label{sec:rank1}
   Denote by $\mr{Ref}^{(1)}(X)$ the space of reflexive rank one sheaves on $X$. Let $i:X^* \to X$ be the natural inclusion.
   Recall, every reflexive sheaf of rank one on a regular variety is invertible (see \cite[Proposition $1.9$]{stabhar}).
   Moreover, every reflexive sheaf on $X$ arises as the pushforward via $i$ of a reflexive sheaf on $X^*$ (see 
   \cite[Proposition $1.6$]{stabhar}). This means that under pushforward by $i$,
   \[i_*: \mr{Pic}(X^*) \to \mr{Ref}^{(1)}(X) \mbox{ sending } \mc{L} \mbox{ to } i_*\mc{L}\]
   is an isomorphism. The group operation on $\mr{Pic}(X^*)$ induces one on $\mr{Ref}^{(1)}(X)$, namely 
   \[M . N:= i_*(i^*M \otimes_{\mo_{X^*}} i^*N) \mbox{ and } M^{\vee}:= i_*((i^*M)^{\vee}). \]

	\begin{thm}\label{thm:ref}
	The morphism \[\phi: \mc{D}(X) \to \mr{Ref}^{(1)}(X)\] sending a $\mb{C}^*$-divisor $D$ to the reflexive sheaf 
	$i_*(\mo_{X^*}(D \cap X^*))$ is an isomorphism of abelian groups.
	\end{thm}
	
	\begin{proof}
 Clearly, this is a group homomorphism. Moreover, as $U$ is integral and regular, $\mc{D}(X)$ is contained in $\mr{Pic}(U)$.
 Since $\mr{Ref}^{(1)}(X)$ is isomorphic to $\mr{Pic}(U)$ as argued above, this means the morphism $\phi$ is injective.
 So it remains to check that $\phi$ is surjective.
 
 Consider $M \in \mr{Ref}^{(1)}(X)$. Note that, the restriction $M|_{X^*}$ is a reflexive sheaf. Since 
 $X^*$ is regular, this implies $M|_{X^*}$ is an invertible sheaf. In other words, 
 \[M|_{X^*} \cong \mo_{X^*}(D^*)\]
 for some divisor $D^*$ on $X^*$. Write $D^*=\sum_i a_iD_i$ as a linear combination of integral 
 divisors $D_i$. If $D_i$ is not horizontal, then by the fiber dimension theorem the scheme theoretic image $E_i$ of 
 $\pi_X|_{D_i}$ is a divisor in $\mb{P}^{\un{\omega}}_{X^*}$. Since $D_i$ is integral, $E_i$ is irreducible.
 This implies $\pi_X^{-1}(E_i)$ is irreducible  (see \S \ref{sec:div}). Hence, $D_i=\pi_X^{-1}(E_i)_{\red}$. 
 In other words, $D_i$ is a $\mb{C}^*$-divisor. 
 	If $D_i$ is horizontal, then by Theorem \ref{thm:ci} there exists a $\mb{C}^*$-curve $D^c_i$ such that 
	$D_i$ is linearly equivalent to $-D^c_i$. Therefore, $D^*$ is linearly equivalent to a divisor obtained as a linear combination of 
	$\mb{C}^*$-divisors. This proves surjectivity of $\phi$ and hence the theorem.
	\end{proof}

\subsection{Dimension two case}
Let $(X,x)$ be an isolated, quasi-homogeneous hypersurface singularity.
Suppose that $\dim X=2$. Denote by $k_x$ the skyscraper sheaf over the singular point $x$ of a one dimensional vector 
space.

\begin{thm}\label{thm:mck}
  Let $M$ be a maximal Cohen-Macaulay $\mo_X$-module of rank, say $r$.
 Then, for a general choice of $r$ sections $(s_1,...,s_r)$ of $M$, we have an exact sequence of the 
 form 
 \begin{equation}\label{eq:mck}
  0 \to \mo_X^{\oplus r-1} \xrightarrow{(s_1,...,s_{r-1})} M \to \mc{L} \to k_x^{\oplus m} \to 0
 \end{equation}
 for some non-negative integer $m$ and $\mc{L}$ is a reflexive sheaf on $X$ of rank $1$ i.e., 
 $\mc{L} \in \mr{Ref}^{(1)}(X)$. In particular, if $C$ denotes the support of the cokernel of the
 morphism $(s_1,...,s_r)$, then $\mc{L}$ is the dual of the ideal sheaf of $C$ in $X$.
\end{thm}

\begin{proof}
 Denote by $\mc{A}$ the cokernel of the morphism 
 \begin{equation}\label{eq:cok}
  (s_1,...,s_r): \mo_X^{\oplus r} \to M.
 \end{equation}
 Note that, $\mc{A}$ is a Cohen-Macaulay module supported in dimension $1$ and of rank one 
 over its support.
Denote by $C$ the support of $\mc{A}$ and $\mc{A}':=\mc{E}xt^1_X(\mc{A},\mo_X)$.
Dualizing \eqref{eq:cok}, we then have the following diagram of short exact sequences:
\[\begin{diagram}
   0 &\rTo& \I_{C|X} &\rTo&\mo_X&\rTo&\mo_C&\rTo0\\
   & & \dTo& &\dTo^{p_1}& &\dTo^{p_2}\\
   0 &\rTo& M^{\vee} &\rTo&\mo_X^{\oplus r}&\rTo^{(t_1,...,t_r)}&\mc{A}'&\rTo0
  \end{diagram}\]
where the morphism $p_2$ sends $1$ to $t_1$ and $p_1$ sends $1$ to the standard basis 
element $e_1 \in \mo_X^{\oplus r}$. Then, the cokernel of $p_1$ is isomorphic $\mo_X^{\oplus r-1}$.
Since $\mc{A}'$ is Cohen-Macaulay, the morphism $p_2$ is injective (the section $t_1$ is torsion-free
over $C$).
By Bertini-type theorem (see \cite[p. $434$]{prag}), $C\backslash \{x\}$ is non-singular. Since any torsion-free 
sheaf on an affine non-singular curve is trivial, we conclude  $\mc{A}'$ is isomorphic to $\mo_C$ over $X^*$.
Taking $t_1=1 \in \Gamma(\mo_C)$, we observe that the cokernel of $p_2$ is of the form $k_x^{\oplus m}$
for some non-negative integer $m$. Using Snake lemma, we get the exact sequence:
\[0 \to \I_{C|X} \to M^{\vee} \to \mo_X^{\oplus r-1} \to k_x^{\oplus m} \to 0\]
Dualizing this sequence and applying \cite[Theorem $3.3.10$]{brun1}, gives us the exact sequence \eqref{eq:mck}.     
This proves the theorem. 
\end{proof}

	\section{Matrix factorization using degeneracy modules}
	Matrix factorization of maximal Cohen-Macaulay modules is hard. However, one can instead study resolutions 
	of the associated degeneracy modules. This is a slightly easier problem. We obtain matrix factorizations 
	using this idea (see Theorem \ref{thm2} and Corollary \ref{thm1}). We then apply this to 
	enumerate the matrix factorization of all Cohen-Macaulay modules arising from $\mb{C}^*$-curves in 
	quasi-homogeneous surfaces 	(see Theorem \ref{thm:Quasi} stated in the introduction and proved in \S \ref{sec:proof}).

	\subsection{Matrix factorization via degeneracy modules}\label{sec:mat}
	Let $M$ be a maximal Cohen-Macaulay $\mo_X$-module of rank $r$
	with no free direct summand (i.e., does not contain $\mo_X$ as a 
	direct summand). Let $\un{s}$ be an $r$-tuple of weakly 
	general sections of $M$. Let $(\mc{A}_{\un{s}},\un{s}')$ be the associated 
	degenerate pair.
	Since $\mc{A}_{\un{s}}$ is a Cohen-Macaulay $\mo_X$-module supported in a dimension $n-2$ subvariety 
	in $\cn$, the depth of $\mc{A}_{\un{s}}$ is $n-2$. By the Auslander-Buchsbaum formula this 
	implies the projective dimension of $\mc{A}_{\un{s}}$ is $2$.
	Then starting with $\un{s}'$ the pair induces an exact sequence of the form:
	\begin{equation}\label{eq07}
		0 \to \mo_{\cn}^{\oplus a} \xrightarrow{A} \mo_{\cn}^{\oplus b} \xrightarrow{B} 
		\mo_{\cn}^{\oplus r} \xrightarrow{\un{s}'} \mc{A}_{\un{s}} \to 0
	\end{equation}
	where $A$ (resp. $B$) is induced by a $b \times a$ (resp. $r \times b$) matrix with entries in 
	$\mo_{\cn}$, which we will also denote by $A$ (resp. $B$) for simplicity of notation.
	In particular, 
	\[Ae_i^{(a)} = \sum\limits_{j=1}^b a_{ji} e_j^{(b)} \mbox{ and } Be_i^{(b)}=\sum\limits_{j=1}^r b_{ji}
	e_j^{(r)},\] 
	where $\{e_i^{(t)}\}_{i=1}^t$ is the standard basis of the free $\mo_{\cn}$-module
	$\mo_{\cn}^{\oplus t}$ for $t \in \{r,a,b\}$. We show:
	
	\begin{thm}\label{thm2}
		Denote by $K$ the $\mo_{\cn}$-submodule of $\mo_{\cn}^{\oplus b}$ consisting of all 
		$m \in \mo_{\cn}^{\oplus b}$ such that $Bm \in \I_X^{\oplus r}$. Then, 
		\begin{enumerate}
			\item $K$ is isomorphic to $\mo_{\cn}^{\oplus r}$, as $\mo_{\cn}$-modules,
			\item fix an isomorphism as in $(1)$ from $\mo_{\cn}^{\oplus r}$ to $K$ given by 
			a $b \times r$-matrix
			 \[A': \mo_{\cn}^{\oplus r} \xrightarrow{\sim} K \subset \mo_{\cn}^{\oplus b}.\]
			 Then,
			(upto change of basis of $\mo_{\cn}^{\oplus r}$) the composition 
			\[\mo_{\cn}^{\oplus r} \xrightarrow{A'} \mo_{\cn}^{\oplus b} \xrightarrow{B} \mo_{\cn}^{\oplus r}
			\mbox{ coincides with } F\mr{Id}_{r \times r}: \mo_{\cn}^{\oplus r} \to \mo_{\cn}^{\oplus r},  \]
			where $F \in \mo_{\cn}$ defines $X$,  
			\item the matrix factorization associated to $M$ is of the form $\left(\mathrm{adj}(A|A')^T, (A|A')^T\right)$, where $(-)^T$ denotes
			transpose of the matrix and $\mathrm{adj}(-)$ denotes the adjoint of the matrix.
		\end{enumerate}
	\end{thm}
	
	Before we prove the theorem, note that by the exact sequence \eqref{eq07}, we have 
	$b=r+a$ (the support of $\mc{A}_{\un{s}}$ is of codimension $2$ in $\cn$). Then, the matrix 
	$(A|A')$ is a $b \times b$-matrix.
	
	\begin{proof}
		Comparing the exact sequences \eqref{eq05} and \eqref{eq07}, we get the following diagram of exact 
		sequences:
		\[\begin{diagram}
			0 &\rTo& \mo_{\cn}^{\oplus a} &\rTo^{A} &\mo_{\cn}^{\oplus b} &\rTo^{B}&
			\mo_{\cn}^{\oplus r} &\rTo^{\un{s}'}& \mc{A}_{\un{s}}& \rTo&0\\
			& && &\dDashto^{\rho'}&\circlearrowleft&\dTo^{\rho}&\circlearrowleft&\dTo^{\mr{id}}&&\\
			& &  0 & \rTo& M^{\vee} &\rTo& \mo_X^{\oplus r} &\rTo^{\un{s}'}&
			\mc{A}_{\un{s}} &\rTo& 0
		\end{diagram}\]
		where the vertical morphism $\rho$ is the natural restriction morphism and the first vertical 
		morphism $\rho'$ is induced by the universal property of kernel.
		Since the last two vertical arrows are surjective then by a simple diagram chase 
		(using the injectivity of the morphism from $M^{\vee}$ to $\mo_X^{\oplus r}$) we conclude
		that morphism $\rho'$ from $\mo_{\cn}^{\oplus b}$ to $M^{\vee}$
		is surjective. Note that, $\rho$ sits in the short exact sequence:
		\[0 \to \mo_{\cn}^{\oplus r} \xrightarrow{F\mr{Id}_{r \times r}} \mo_{\cn}^{\oplus r} 
		\xrightarrow{\rho} \mo_X^{\oplus r} \to 0. \]
		Using the Snake lemma applied to the above diagram of exact sequence, this gives us the following
		exact sequence:
		\begin{equation}\label{eq08}
			0 \to \mo_{\cn}^{\oplus a} \oplus \mo_{\cn}^{\oplus r} \xrightarrow{(A|A')} \mo_{\cn}^{\oplus b}
			\xrightarrow{\rho'} M^{\vee} \to 0 
		\end{equation}
		where the composition 
		\[\mo_{\cn}^{\oplus r} \xrightarrow{A'}  \mo_{\cn}^{\oplus b} \xrightarrow{B}  \mo_{\cn}^{\oplus r}
		\mbox{ coincides with } F\mr{Id}_{r \times r}:  \mo_{\cn}^{\oplus r} \to  \mo_{\cn}^{\oplus r}.\]
		This proves parts ($1$) and ($2$) of the theorem (identify $K$ with the image of $A'$).
		As mentioned above $b=r+a$. Dualizing \eqref{eq08}, we get the exact sequence:
		\begin{equation}\label{eq09}
			0 \to \mo_{\cn}^{\oplus b} \xrightarrow{(A|A')^T} \mo_{\cn}^{\oplus b} \to \mc{E}xt^1_{\cn}(M^{\vee},\mo_{\cn})
			\to 0.
		\end{equation}
		Since $F$ annihilates $M^\vee$ (as $M^{\vee}$ is supported in $X$), we have by \cite[Lemma $1.2.4$]{brun1}
		\[\mc{E}xt^1_{\cn}(M^\vee,\mo_{\cn}) \cong \Hc_{\cn}(M^{\vee},\mo_X) \cong \Hc_X(M^\vee,\mo_X),\]
		where the last isomorphism follows from adjunction of $\Hc$-functor.
		Since $M$ is a maximal Cohen-Macaulay $\mo_X$-module, it is in particular reflexive.
		Therefore, the double dual $M^{\vee \vee}$ of $M$ is isomorphic to $M$.
		Hence, $\mc{E}xt^1_{\cn}(M^{\vee},\mo_{\cn}) \cong M$ and \eqref{eq09} gives a projective 
		resolution of $M$. In other words, $\left(\mathrm{adj}(A|A')^T, (A|A')^T\right)$ is a matrix factorization. This proves the theorem.
	\end{proof}
	
% 	\section{Matrix factorizations for quasi-homogeneous and cusp singularities}\label{sec:quas}
% 	For normal hypersurface singularities of dimension two, the cohomologically special modules are very well understood geometrically (see~\cite{BoRo}). In this and the following sections, we use Theorem~\ref{thm1} to construct families of matrix factorizations for families singularities[EDIT?? refer to Theorem and quasi-hom].
% 	
	\subsection{Generalized Wunram modules}\label{sec:gen-wun}
	Following~\cite{BoRo}, a maximal Cohen-Macaulay $\mo_X$-module $M$ of rank $1$ is 
	called \emph{generalized Wunram} if for a general choice of 
	section $s$ of $M$, the cokernel of the natural morphism from $\mo_X$ to $M$, defined by multiplication with $s$,
	is isomorphic to $\mo_D$ for a non-singular subvariety $D \subset X$ of codimension $1$.

	\subsection{Projective resolution of the degeneracy module}
	Let $M$ be a rank one generalized Wunram module, $s \in M$ a general section and $D$ be the 
	associated degeneracy locus.
	In particular, we have a short exact sequence of the form:
	\begin{equation}\label{eq01}
		0 \to \mo_X \xrightarrow{.s} M \to \mo_D \to 0,
	\end{equation}
	where $D$ is a non-singular subvariety in $X$ of codimension $1$.
	Dualizing this exact sequence we get a short exact sequence of the form:
	\begin{equation}\label{eq02}
		0 \to M^{\vee} \to \mo_X \to \mc{E}xt^1_X(\mo_D,\mo_X) \to 0
	\end{equation}
	Note that, $\mc{E}xt^1_X(\mo_D,\mo_X)$ is a Cohen-Macaulay $\mo_X$-module 
	supported on $D$ and is of rank one over its support.
	Now, a rank one maximal Cohen-Macaulay module over a smooth affine variety is trivial.
	Hence, $\mc{E}xt^1_X(\mo_D,\mo_X) \cong \mo_D$.
	We now produce a projective resolution of $\mo_D$.
	Since $D$ is non-singular there exists $f, g \in \mo_{\cn}$ such that 
	the ideal of $D$ (in $\cn$) is generated by 
	$f$ and $g$ (regular local rings are complete intersection rings).
	We then have the Koszul resolution:
	
	\begin{prop}\label{prop:koz}
		The projective resolution of $\mo_D$ is given by 
		\[0 \to \mo_{\cn} \xrightarrow{A} \mo_{\cn}^{\oplus 2} \xrightarrow{B} \mo_{\cn} \to \mo_D \to 0\]
		where $Ae := -fe_1+ge_2$, $Be_1:= g$ and $Be_2=f$ with $e$ (resp. $\{e_1, e_2\}$)
		the standard basis of $\mo_{\cn}$ (resp. $\mo_{\cn}^{\oplus 2}$).
	\end{prop}

	\begin{cor}\label{thm1}
		Let $X$ be a normal hypersurface singularity (not necessarily isolated) of any dimension.
		Let $M$ be a rank one generalized Wunram module,
		$s \in M$ a general section and 
		$D$ the degeneracy locus associated to the pair $(M,s)$, which is non-singular as $M$ is generalized Wunram
		of rank one. 
		Then, the matrix factorization associated to $M$ is the pair $(\mr{adj}(C),C)$ where $C$ is the matrix
		\[C:= \begin{pmatrix}
			-f & g\\
			h_1 & h_2
		\end{pmatrix}
		\]
		$f, g \in \mo_{\cn}$ defines the non-singular variety $D$ in $\cn$ and $X$ is defined by a 
		regular function of the form $F:=h_1g+h_2f$ (as $D \subset X$ we have $F \in (f,g)$).
	\end{cor}
	
	\begin{proof}
		Translating into the notations of Theorem \ref{thm2}, we have $a=1, b=2$ and $r=1$.
		The morphisms $A$ and $B$ are defined in Proposition \ref{prop:koz}.
		We now need to compute $K$ and $A'$ from Theorem \ref{thm2}.
		Recall, 
		\[K=\{ a_1e_1+a_2e_2 | a_1g+a_2f \in \I_X\} \mbox{ where }\]
		$\I_X$ is the ideal of $X$ in $\cn$ generated by, say $F$.
		Of course, since $D \subset X$, there exists $h_1, h_2 \in \mo_{\cn}$ such that 
		$F=h_1g+h_2f$. In other words, $h_1e_1+h_2e_2 \in K$.
		We claim that $K$ is generated as an $\mo_{\cn}$-module by $h_1e_1+h_2e_2$.
		Indeed, since $K \cong \mo_{\cn}$ (Theorem \ref{thm2}), it is generated by a single 
		element, say $h_1'e_1+h_2'e_2 \in \mo_{\cn}^{\oplus 2}$. Then, there exists 
		$\lambda \in \mo_{\cn}$ such that \[\lambda(h_1'e_1+h_2'e_2)=h_1e_1+h_2e_2.\]
		Applying the $\mo_{\cn}$-linear morphism $B$, we have 
		\begin{equation}\label{eq10}
			\lambda B(h_1'e_1+h_2'e_2)= B(\lambda(h_1'e_1+h_2'e_2))= B(h_1e_1+h_2e_2)=F.
		\end{equation}
		Since $h_1'e_1+h_2'e_2 \in K$, we have $B(h_1'e_1+h_2'e_2)=\lambda' F$ for some 
		$\lambda' \in \mo_{\cn}$. Substituting in \eqref{eq10} this implies $\lambda \lambda'=1$
		i.e., $\lambda$ is a unit in $\mo_{\cn}$. This proves our claim that 
		$K$ is generated as an $\mo_{\cn}$-module by $h_1e_1+h_2e_2$.
		Then, we can take the morphism 
		\[A': \mo_{\cn} \xrightarrow{\sim} K \subset \mo_{\cn}^{\oplus 2} 
		\mbox{ sending } 1 \mbox{ to } h_1e_1+h_2e_2.\]
		This satisfies the condition that the composition $B \circ A'=F \times \mr{\Id}$.
		By Theorem \ref{thm2} the matrix factorization of $M$ is of the form $\left(\mathrm{adj}(A|A')^T, (A|A')^T\right)$ where   
		\[
		(A|A')=\begin{pmatrix}[c|c]
			-f &  h_1\\
			g & h_2
		\end{pmatrix}
		\mbox{, so } 
		 (A|A')^T=\begin{pmatrix}
			-f &  g\\
			h_1 & h_2
		\end{pmatrix}\]
		This proves the corollary.
	\end{proof}

% 	\begin{rem}
% 		Given $f$ a quasi-homogeneous surface singularity with an isolated critical point. Our main interest is to use smooth $\CC^*$-curves to produce matrix factorizations of $f$. Since the existence of smooth $\CC^*$-curves is completely determined by the weights (at least one weight should be one) and by the work of Xu and Yau~\cite{xu} the weights determine the topological type of $f$, thus we can reduce our problem by topological trivial deformations of $f$. This is done in the following section. 
% 	\end{rem}

	\subsection{Matrix factorization for topological trivial deformations}\label{sec:type}
	Orlik and Wagreich~\cite{orli} and Arnold~\cite{arn1} showed that an isolated, quasi-homogeneous surface 
	singularity can be 
	can be deformed into one of the following seven classes below keeping the link differentially constant
	\begin{center}
	 \begin{tabular}{| m{1.8cm} | m{13cm}|}
	 \hline
	 Type & Defining polynomial \\
	 \hline
  $\mr{I}_{p,q,r}$ & $F(X_1,X_2,X_3):= X_1^{p}+X_2^{q}+X_3^{r}$\\
  \hline
		$\mr{II}_{p,q,r}$ & $F(X_1,X_2,X_3):= X_1^{p}+X_2^{q}+X_2X_3^{r}$ with $q >1$\\
		\hline
		$\mr{III}_{p,q,r}$ & $F(X_1,X_2,X_3):= X_1^{p}+X_3X_2^{q}+X_2X_3^{r}$ with $q > 1$ and $r>1$\\
		\hline
		$\mr{IV}_{p,q,r}$ & $F(X_1,X_2,X_3):= X_1^{p}+X_3X_2^{q}+X_1X_3^{r}$ with $p>1$\\
		\hline
		 $\mr{V}_{p,q,r}$ & $F(X_1,X_2,X_3):= X_2X_1^{p}+X_3X_2^{q}+X_1X_3^{r}=0$\\ 
		 \hline
$\mr{VI}_{p,q,r,b_2,b_3}$ & $F(X_1,X_2,X_3):=X_1^{p}+X_1X_2^{q}+X_1X_3^{r}+X_2^{b_2}X_3^{b_3}$ with $(p-1)(qb_3+rb_2)=pqr$\\ 
\hline
$\mr{VII}_{p,q,r,b_2,b_3}$ & $F(X_1,X_2,X_3):= X_2X_1^p+X_1X_2^q+X_1X_3^r+ X_2^{b_2}X_3^{b_3}$ with $(p-1)(qb_3+rb_2)=r(pq-1)$\\
		\hline
	 \end{tabular}
	 \captionof{table}{Quasi-homogeneous singularity types}\label{tabl1}
	\end{center}
	Xu and Yau~\cite{xu} proved that the above deformation is in fact a topological trivial deformation. 
	Furthermore, the topological type of quasi-homogeneous singularities
	determine and is determined by its weights. We now use Corollary~\ref{thm1} to produce the matrix factorizations corresponding to all 
	rank one generalized Wunram modules.
	
	\subsection{Proof of Theorem \ref{thm:Quasi}}\label{sec:proof}
	Given $c_1, c_2 \in \mb{C}$ and $k,n,m \in \mb{Z}_{>0}$ denote by 
	\[G_{(c_1,c_2,n)}(Z_1,Z_2):=c_1 Z_1^{n} -c_2^{n}Z_2.\]
 and $S_{(c_1,c_2,n,m)}(Z_1,Z_2)$ defined in Theorem \ref{thm:Quasi}.	Note that,
	\begin{equation}
		\label{eq:ProdElem}
		Z_3^k G_{(c_1,c_2,n)}(Z_1,Z_2)S_{(c_1,c_2,n,m)}(Z_1,Z_2)=Z_3^k \left(\frac{c_1^mZ_1^{mn}}{c_2^{mn}}-Z_2^m\right).
	\end{equation}
Let $X$ be a quasi-homogeneous surface singularity defined by a quasi-homogeneous polynomial $F(X_1,X_2,X_3)$ from the list 
in Table \ref{tabl1} above. By assumption, the weights of $X$ is $(1, \omega_2, \omega_3)$. 
Take a point $(a,b,c) \in X$  with $a \neq 0$.
The associated $\CC^*$-curve, denoted $W_{a,b,c}$, is given by the following parametrization:
\[n \colon \mb{C}^* \to X\, \mbox{ such that }\, \lambda \mapsto (a\lambda,b\lambda^{\omega_2}, c\lambda^{\omega_3}).\]
Note that, $W_{a,b,c}$ is the zero locus (in $\mb{C}^3$) of the polynomials
\begin{equation*}
	G_{(b,a,\omega_2)}(X_1,X_2)=X_1^{\omega_2}b-X_2a^{\omega_2} \quad \text{and} \quad 	G_{(c,a,\omega_3)}(X_1,X_3)=X_1^{\omega_3}c-X_3a^{\omega_3}.
\end{equation*}
By Corollary~\ref{thm1} we only need to find $h_1,h_2 \in \CC[X_1,X_2,X_3]$ such that 
\begin{equation*}
	F= G_{(c,a,\omega_3)}(X_1,X_3)h_1+G_{(b,a,\omega_2)}(X_1,X_2)h_2.
\end{equation*}

{\bf Type $\mr{I}_{p,q,r}$:} In this case $F = X_1^{p}+X_2^{q}+X_3^{r}$. 
By equation~\eqref{eq:ProdElem},
\[	G_{(b,a,\omega_2)}(X_1,X_2)S_{(b,a,\omega_2,q)}(X_1,X_2) + G_{(c,a,\omega_3)}(X_1,X_3)S_{(c,a,\omega_3,r)}(X_1,X_3) = \frac{b^qX_1^{q\omega_2}}{a^{q\omega_2}}-X_2^q + \frac{c^rX_1^{r\omega_3}}{a^{r\omega_3}}-X_3^r.
\]
As $F$ is quasi-homogeneous we have $p=p\omega_1 = q\omega_2= r\omega_3$. Moreover, as 
$(a,b,c) \in X$, we have $a^p+b^q+c^r=0$. Therefore,
\begin{equation*}
	\frac{b^qX_1^{q\omega_2}}{a^{q\omega_2}}+ \frac{c^rX_1^{r\omega_3}}{a^{r\omega_3}} = X_1^p\left(\frac{b^q}{a^{q\omega_2}} + \frac{c^r}{a^{r\omega_3}}\right) = X_1^p\left(\frac{b^q+c^r}{a^{p}}\right) = -X_1^p.
\end{equation*}
Thus,
$	G_{(b,a,\omega_2)}(X_1,X_2)S_{(b,a,\omega_2,q)}(X_1,X_2) + G_{(c,a,\omega_3)}(X_1,X_3)S_{(c,a,\omega_3,r)}(X_1,X_3) =- F$.
In particular, $h_1:= S_{(c,a,\omega_3,r)}(X_1,X_3)$ and $h_2:= S_{(b,a,\omega_2,q)}(X_1,X_2)$.
This prove the matrix factorization in this case.

{\bf Type $\mr{II}_{p,q,r}$:} In this case $F = X_1^{p}+X_2^{q}+X_2X_3^{r}$. 
By equation~\eqref{eq:ProdElem},
\begin{align*}
	&G_{(b,a,\omega_2)}(X_1,X_2)\left(S_{(b,a,\omega_2,q)}(X_1,X_2) + X_3^rS_{(b,a,\omega_2,1)}(X_1,X_2) \right) + \frac{bX_1^{\omega_2}}{a^{\omega_2}} G_{(c,a,\omega_3)}(X_1,X_3)S_{(c,a,\omega_3,r)}(X_1,X_3)\\
	&= \frac{b^qX_1^{q\omega_2}}{a^{q\omega_2}}-X_2^q + X_3^r\frac{bX_1^{\omega_2}}{a^{\omega_2}}-X_2X_3^r +  \left(\frac{bX_1^{\omega_2}}{a^{\omega_2}}\right)\frac{c^rX_1^{r\omega_3}}{a^{r\omega_3}}-\left(\frac{bX_1^{\omega_2}}{a^{\omega_2}}\right)X_3^r\\
	&=\frac{b^qX_1^{q\omega_2}}{a^{q\omega_2}}-X_2^q - X_2X_3^r +  \frac{bc^rX_1^{r\omega_3+\omega_2}}{a^{r\omega_3+\omega_2}}-X_3^r.
\end{align*}
Arguing as before ($F$ is quasi-homogeneous), we have 
\begin{equation*}
	\frac{b^qX_1^{q\omega_2}}{a^{q\omega_2}} + \frac{bc^rX_1^{r\omega_3+\omega_2}}{a^{r\omega_3+\omega_2}} = X_1^p \left( \frac{b^q+bc^r}{a^p}\right) = -X_1^p.
\end{equation*}
Therefore (use $p=p\omega_1 = q\omega_2= r\omega_3+\omega_2$),
\begin{align*}
	&G_{(b,a,\omega_2)}(X_1,X_2)\left(S_{(b,a,\omega_2,q)}(X_1,X_2) + X_3^rS_{(b,a,\omega_2,1)}(X_1,X_2) \right) + \frac{bX_1^{\omega_2}}{a^{\omega_2}} G_{(c,a,\omega_3)}(X_1,X_3)S_{(c,a,\omega_3,r)}(X_1,X_3)\\
	&= -X_2^q - X_2X_3^r - X_1^p.
\end{align*}
This gives the matrix factorization in this case.

{\bf Type $\mr{III}_{p,q,r}$:} In this case $F = X_1^{p}+X_3X_2^{q}+X_2X_3^{r}$.
Arguing as before, we have using \eqref{eq:ProdElem},
\begin{align*}
	&G_{(b,a,\omega_2)}(X_1,X_2)\left(X_3S_{(b,a,\omega_2,q)}(X_1,X_2)+X_3^rS_{(b,a,\omega_2,1)}(X_1,X_2) \right)\\
	&+ G_{(c,a,\omega_3)}(X_1,X_3) \left(\frac{bX_1^{\omega_2}}{a^{\omega_2}}S_{(c,a,\omega_3,r)}(X_1,X_3) + \frac{b^qX_1^{q\omega_2}}{a^{q\omega_2}}S_{(c,a,\omega_3,1)}(X_1,X_3)\right)\\
	&= \frac{b^qX_3X_1^{q\omega_2}}{a^{q\omega_2}}-X_3X_2^q + X_3^r\frac{bX_1^{\omega_2}}{a^{\omega_2}}-X_2X_3^r + \frac{bc^rX_1^{r\omega_3+\omega_2}}{a^{r\omega_3+\omega_2}}-\left(\frac{bX_1^{\omega_2}}{a^{\omega_2}}\right)X_3^r +\frac{b^qcX_1^{\omega_3+q\omega_2}}{a^{\omega_3+q\omega_2}}-\frac{b^qX_1^{q\omega_2}}{a^{q\omega_2}}X_3\\
	&=-X_3X_2^q -X_2X_3^r + \frac{bc^rX_1^{r\omega_3+\omega_2}}{a^{r\omega_3+\omega_2}} +\frac{b^qcX_1^{\omega_3+q\omega_2}}{a^{\omega_3+q\omega_2}}\, \, \mbox{ and }\\
	& \frac{bc^rX_1^{r\omega_3+\omega_2}}{a^{r\omega_3+\omega_2}} +\frac{b^qcX_1^{\omega_3+q\omega_2}}{a^{\omega_3+q\omega_2}}=X_1^p \left(\frac{bc^r+b^qc}{a^{p}}\right)  = -X_1^p.
\end{align*}
(use $p=p\omega_1 = q\omega_2+\omega_3= r\omega_3+\omega_2$ for the last equality).
This proves the matrix factorization in this case.

{\bf Type $\mr{IV}_{p,q,r}$:} In this case $F = X_1^{p}+X_3X_2^{q}+X_1X_3^{r}$. 
Arguing as before, using \eqref{eq:ProdElem} we have 
\begin{align*}
	&G_{(b,a,\omega_2)}(X_1,X_2)\left(X_3S_{(b,a,\omega_2,q)}(X_1,X_2)\right)\\
	&+G_{(c,a,\omega_3)}(X_1,X_3) \left(XS_{(c,a,\omega_3,r)}(X_1,X_3) + \frac{b^qX_1^{q\omega_2}}{a^{q\omega_2}}S_{(c,a,\omega_3,1)}(X_1,X_3)\right)\\
	&= -X_3X_2^q+\frac{c^rX_1^{r\omega_3+1}}{a^{r\omega_3}}-X_1X_3^r +\frac{b^qcX_1^{\omega_3+q\omega_2}}{a^{\omega_3+q\omega_2}}\, \mbox{ and }\\
	& \frac{c^rX_1^{r\omega_3+1}}{a^{r\omega_3}}+\frac{b^qcX_1^{\omega_3+q\omega_2}}{a^{\omega_3+q\omega_2}}=X_1^p \left(\frac{ac^r+b^qc}{a^{p}}\right)  = -X_1^p.
\end{align*}
(use $p=p\omega_1 = q\omega_2+\omega_3= r\omega_3+\omega_1=r\omega_3+1$ for the last equality).
This proves the matrix factorization in this case.

{\bf Type $\mr{V}_{p,q,r}$:} In this case $F(X_1,X_2,X_3) = X_2X_1^{p}+X_3X_2^{q}+X_1X_3^{r}$. 
Arguing as before using equation~\eqref{eq:ProdElem} we have,
\begin{align*}
	&G_{(b,a,\omega_2)}(X_1,X_2)\left(X_3S_{(b,a,\omega_2,q)}(X_1,X_2)+X_1^pS_{(b,a,\omega_2,1)}(X_1,X_2) \right)\\
	&+ G_{(c,a,\omega_3)}(X_1,X_3) \left(XS_{(c,a,\omega_3,r)}(X_1,X_3) + \frac{b^qX_1^{q\omega_2}}{a^{q\omega_2}}S_{(c,a,\omega_3,1)}(X_1,X_3)\right)\\
	&=-X_3X_2^q + \frac{bX_1^{\omega_2+p}}{a^{\omega_2}}-X_2X_1^p+\frac{c^rX_1^{r\omega_3+1}}{a^{r\omega_3}}-X_1X_3^r +\frac{b^qcX_1^{\omega_3+q\omega_2}}{a^{\omega_3+q\omega_2}}\, \mbox{ and }\\
	& \frac{bX_1^{\omega_2+p}}{a^{\omega_2}}+\frac{c^rX_1^{r\omega_3+1}}{a^{r\omega_3}}+\frac{b^qcX_1^{\omega_3+q\omega_2}}{a^{\omega_3+q\omega_2}}=X_1^{\omega_2+p}\left(\frac{ba^p+ac^r+b^qc}{a^{\omega_2+p}}\right)=0.
\end{align*}
(use $p+\omega_2 = q\omega_2+\omega_3= r\omega_3+\omega_1=r\omega_3+1$ for the last equality). 
This proves the 
matrix factorization in this case.

{\bf Type $\mr{VI}_{p,q,r,b_2,b_3}$:} In this case $F = X_1^{p}+X_1X_2^{q}+X_1X_3^{r}+X_2^{b_2}X_3^{b_3}$.
Arguing as before using \eqref{eq:ProdElem} we have 
\begin{align*}
	&G_{(b,a,\omega_2)}(X_1,X_2)\left(XS_{(b,a,\omega_2,q)}(X_1,X_2)+X_3^{b_3}S_{(b,a,\omega_2,b_2)}(X_1,X_2) \right)\\
	&+G_{(c,a,\omega_3)}(X_1,X_3) \left(XS_{(c,a,\omega_3,r)}(X_1,X_3) + \frac{b^{b_2}X_1^{b_2\omega_2}}{a^{b_2\omega_2}}S_{(c,a,\omega_3,b_3)}(X_1,X_3)\right)\\
	&=\frac{b^qX_1^{q\omega_2+1}}{a^{q\omega_2}}-X_1X_2^q + X_3^{b_3}\frac{b^{b_2}X_1^{b_2\omega_2}}{a^{b_2\omega_2}}-X_2^{b_2}X_3^{b_3}+\frac{c^rX_1^{r\omega_3+1}}{a^{r\omega_3}}-X_1X_3^r +\frac{b^{b_2}X_1^{b_2\omega_2}}{a^{b_2\omega_2}} \left( \frac{c^{b_3}X_1^{\omega_3 b_3}}{a^{\omega_3 b_3}}-X_3^{b_3}\right)\\
	&=\frac{b^qX_1^{q\omega_2+1}}{a^{q\omega_2}}-X_1X_2^q -X_2^{b_2}X_3^{b_3}+\frac{c^rX_1^{r\omega_3+1}}{a^{r\omega_3}}-X_1X_3^r +\frac{c^{b_3}b^{b_2}X_1^{b_2\omega_2+\omega_3 b_3}}{a^{b_2\omega_2+\omega_3 b_3}}\, \mbox{ and }\\
	& \frac{b^qX_1^{q\omega_2+1}}{a^{q\omega_2}}+\frac{c^rX_1^{r\omega_3+1}}{a^{r\omega_3}}+\frac{c^{b_3}b^{b_2}X_1^{b_2\omega_2+\omega_3 b_3}}{a^{b_2\omega_2+\omega_3 b_3}}=X_1^p\left(\frac{ab^q}{a^{q\omega_2+1}}+\frac{ac^r}{a^{r\omega_3+1}}+\frac{c^{b_3}b^{b_2}}{a^{b_2\omega_2+\omega_3 b_3}}\right)=-X_1^p
\end{align*}
(use $p = 1+q\omega_2= r\omega_3+1=b_2\omega_2+b_3\omega_3$ for the last equality).
This proves the matrix factorization in this case.

{\bf Type $\mr{VII}_{p,q,r,b_2,b_3}$:} In this case $F = X_2X_1^{p}+X_1X_2^{q}+X_1X_3^{r}+X_2^{b_2}X_3^{b_3}$. 
Arguing as before using \eqref{eq:ProdElem} we have,
\begin{align*}
	&G_{(b,a,\omega_2)}(X_1,X_2)\left(XS_{(b,a,\omega_2,q)}(X_1,X_2)+X_3^{b_3}S_{(b,a,\omega_2,b_2)}(X_1,X_2)+X_1^pS_{(b,a,\omega_2,1)}(X_1,X_2) \right)\\
	&+G_{(c,a,\omega_3)}(X_1,X_3) \left(XS_{(c,a,\omega_3,r)}(X_1,X_3) + \frac{b^{b_2}X_1^{b_2\omega_2}}{a^{b_2\omega_2}}S_{(c,a,\omega_3,b_3)}(X_1,X_3)\right)\\
	&=\frac{b^qX_1^{q\omega_2+1}}{a^{q\omega_2}}-X_1X_2^q + X_3^{b_3}\frac{b^{b_2}X_1^{b_2\omega_2}}{a^{b_2\omega_2}}-X_2^{b_2}X_3^{b_3}+\frac{bX_1^{\omega_2+p}}{a^{\omega_2}}-X_2X_1^p + 	\frac{c^rX_1^{r\omega_3+1}}{a^{r\omega_3}}\\
	&-X_1X_3^r +\frac{b^{b_2}X_1^{b_2\omega_2}}{a^{b_2\omega_2}} \left( \frac{c^{b_3}X_1^{\omega_3 b_3}}{a^{\omega_3 b_3}}-X_3^{b_3}\right)\\
	&=\frac{b^qX_1^{q\omega_2+1}}{a^{q\omega_2}}-X_1X_2^q-X_2^{b_2}X_3^{b_3}+\frac{bX_1^{\omega_2+p}}{a^{\omega_2}}-X_2X_1^p + 	\frac{c^rX_1^{r\omega_3+1}}{a^{r\omega_3}}-X_1X_3^r +\frac{c^{b_3}b^{b_2}X_1^{b_2\omega_2+\omega_3 b_3}}{a^{b_2\omega_2+\omega_3 b_3}}.\\
\end{align*}
Moreover, using $\omega_2+p=1+q\omega_2= r\omega_3+1=b_2\omega_2+b_3\omega_3$ we have 
\begin{equation*}
	\frac{b^qX_1^{q\omega_2+1}}{a^{q\omega_2}}+\frac{bX_1^{\omega_2+p}}{a^{\omega_2}}+\frac{c^rX_1^{r\omega_3+1}}{a^{r\omega_3}}+\frac{c^{b_3}b^{b_2}X_1^{b_2\omega_2+\omega_3 b_3}}{a^{b_2\omega_2+\omega_3 b_3}}=X_1^{q\omega_2+1} \left(\frac{ab^q+a^p b+ac^r+c^{b_3}b^{b_2}}{a^{q\omega_2+1}}\right)=0	
\end{equation*}
This proves the matrix factorization in this case and hence the theorem. \qed
		
	\begin{rem}
		Notice that in the case of $F=X_1^3+X_2^3+X_3^3$, our computation recovers the matrix factorization computed by Laza, Pfister and Popescu~\cite{lazam}. 
	\end{rem}
	
% 	\begin{rem}
% 		Given a general $F=X_1^p+X_2^q+X_3^r$ one could hope to compute all the matrix factorization 
% 		using only smooth $\CC^*$-curves. Nevertheless, it is important to highlight that the $\CC^*$-curves depends only on the topological data. Thus, any possible classification will be a topological classification. For example by \cite{connex}, the singularity $X_1^2+X_2^3+X_3^7$ has non-flat maximal Cohen-Macaulay modules. Therefore, such maximal Cohen-Macaulay modules are impossible to recover using only the topology of the link.
% 	\end{rem}
% 	

\begin{rem}\label{rem:comp}
	For the sake of completeness we now consider the case when $a=0$. 
	For simplicity we consider the polynomial of type $\mr{I}_{p,q,r}$, the remaining 
	cases follow similarly. To fix notation, $F = X_1^{p}+X_2^{q}+X_3^{r}$ with weights $(\omega_1, \omega_2, \omega_3)$ and $V$ is the surface defined by $F$. Let $(a,b,c) \in V(p,q,r)$ with $a=0$. 
	Since the point $(0,b,c)$ is different from the origin and it is a zero of $F$, thus $b$ and $c$ are both non-zero. The $\CC^*$-curve, denoted $W_{b,c}$,
	associated to the point $(0,b,c)$ is given by the parametrization
	\[		n \colon \mb{C}^*  \to X\, \mbox{ such that }
		\lambda \mapsto (0,b\lambda^{\omega_2}, c\lambda^{\omega_3}).\]
	This $\CC^*$-curve is smooth if and only if $\omega_2=1$ or $\omega_3=1$ (upto reparametrization). 
 Without loss of generality suppose that $\omega_2=1$. 
 Under this assumption the $\CC^*$-curve given by the point $(0,b,c)$ is cut out by the polynomials
	\begin{equation*}
		f=X_1 \quad \text{and} \quad 	G_{(c,b,\omega_3)}(X_2,X_3)=cX_2^{\omega_3}-b^{\omega_3}X_3.
	\end{equation*}
	By equation~\eqref{eq:ProdElem},
	\begin{equation*}
		X_1(-X_1^{p-1})+G_{(c,b,\omega_3)}(X_2,X_3)S_{(c,b,\omega_3,r)}(X_2,X_3) =-X_1^p+ \frac{c^rX_2^{r\omega_3}}{b^{r\omega_3}}-X_3^r.
	\end{equation*}
	By assumption, $a=0$ and $b^q+c^r=0$. Therefore,
	\begin{equation*}
		-X_1^p+ \frac{c^rX_2^{r\omega_3}}{b^{r\omega_3}}-X_3^r = -X_1^p -X_2^q-X_3^r.
	\end{equation*}
	Let $M$ be the maximal Cohen-Macaulay module corresponding to the degeneracy locus $W_{b,c}$ 
	(see Proposition \ref{prop:corr}).
	Using Corollary \ref{thm1}, we conclude that the matrix factorization for $M$ is:
	\begin{equation*}
		\begin{pmatrix}
			-X & b^{\omega_3}X_3-cX_2^{\omega_3}\\S_{(c,b,\omega_3,r)}(X_2,X_3)&X_1^{p-1}  \end{pmatrix}.
	\end{equation*}
	\end{rem} 

	\subsection{Conjecture of Etingof-Ginzburg}
	Take $\Phi:= X_1X_2X_3-X_2X_1X_3$. Then, $\mf{A}(\Phi)=\mb{C}[X_1,X_2,X_3]$ (see \cite[Example $1.3.3$]{ginz2}). 
	We prove:
	
	\begin{thm}\label{thm:conj}
	 Let $\Psi \in \mf{A}(\Psi)$  be one of polynomials mentioned in 
	Table \ref{tabl1} such that one of the weights is one. Then, to any point module on $\mf{B}(\Phi, \Psi)$ one 
	can naturally associate a matrix factorization.
	\end{thm}

	\begin{proof}
	Denote by $X$ the hypersurface defined by $\Psi$.
	 	Consider a point $(a,b,c) \in X$ with $a \not= 0$. Denote by $k(a,b,c)$ the residue field associated to the point $(a,b,c)$. Note that, 
	$k(a,b,c)$ is a point module. 
	Then, by Theorem \ref{thm:Quasi} we naturally associate to the  point module $P:= k(a,b,c)$ 
	a matrix factorization $M(P)=(M(P)_+,M(P)_-)$. Moreover, every point module is a direct sum of copies of 
	such residue fields i.e.,
	any point module $P$ is of the form:
	\[ P:= \bigoplus_{i \in I} k(a_i,b_i,c_i)^{\oplus m_i},\, \mbox{ where } (a_i,b_i,c_i) \in X^* \mbox{ and } m_i>0.\]
	Denote by $P_i$ the point module $k(a_i,b_i,c_i)$ and by $M(P_i):=(M(P_i)_+,M(P_i)_-)$ the corresponding 
	matrix factorization.  Denote by $M(P)_+$ (resp. $M(P)_-$) the matrix with diagonal entries
	$m_i$-copies of $M(P_i)_+$ (resp. $M(P_i)_-$) as $i$ varies along the entries in $I$. Then,
	the matrix factorization associated to $P$ is $M(P):= (M(P)_+,M(P)_-)$. This proves the theorem.
	\end{proof}

\section{More examples: cusps and non-isolated singularities}

	In this section we obtain the matrix factorization for certain cusp singularities and non-isolated singularities.

	\subsection{Cusp singularities} \label{sec:cusp}
	Let 
	\begin{equation*}
		F(X_1,X_2,X_3)= X_1^{(r-2)q}+X_2^q+X_3^r+\tau X_1X_2X_3,
	\end{equation*}
	with $\tau \in \CC^*$ and $r \geq 3$. Denote by $X$ the surface defined by $F$. 
	Let $\omega \in \mb{C}$ such that $\omega^{r-1}=1/\tau$. Take a point $(a,b,c)\in \CC^3$ 
	different from the origin such that 
	\begin{equation}\label{eq:ass}
	 a^{q(r-2)}+b^q=0\, \mbox{ and }\, c(c^{r-1}+ab)=0.
	\end{equation}
	 Consider the $\mb{C}^*$-curve, denoted by $W_{a,b,c}$,
	given by the parametrization:
	\[n \colon \mb{C}^* \to X \mbox{ such that }
		\lambda \mapsto (a\lambda\omega,b\lambda^{r-2}\omega^{r-2}, c\lambda).\]
		Note that, the morphism $n$ indeed maps to $X$ because
	\begin{align*}
	F(n(\lambda))&=  (a\lambda\omega)^{(r-2)q}+(b\lambda^{r-2}\omega^{r-2})^q+(c\lambda)^r+\frac{1}{\omega^{r-1}} (a\lambda\omega)(b\lambda^{r-2}\omega^{r-2})(c\lambda)\\
&=(\lambda\omega)^{(r-2)q} \left( a^{(r-2)q} +b^q \right) + \lambda^r\left(c^r +abc \right) = 0 
	\end{align*}
	where the last equality follows from \eqref{eq:ass}. Let $M_{a,b,c}$ be the maximal Cohen-Macaulay 
	$\mo_X$-module associated to the degeneracy locus $W_{a,b,c}$ (see Proposition \ref{prop:corr}). We prove:

\begin{thm}
	\label{th:Cusp}
	The matrix factorization associated to $M_{a,b,c}$ is given by
	\begin{equation*}
        \begin{pmatrix}
			G_{(c,a\omega,1)}(X_1,X_3) & -G_{ (b,a,r-2)}(X_1,X_2)\\
			S_{(b,a,r-2,q)(X_1,X_2)} + \frac{cX_1^2}{a\omega^r}S_{(b,a,r-2,1)(X_1,X_2)}	& S_{(c,a\omega,1,r)}(X_1,X_3)+\frac{X_1X_2}{\omega^{r-1}}S_{(c,a\omega,1,1)}(X_1,X_3)  \end{pmatrix},
	\end{equation*}
	where $G_{(c_1,c_2,n)}(Z_1,Z_2):=c_1 Z_1^{n} -c_2^{n}Z_2$ and 
		\[S_{(c_1,c_2,n,m)}(Z_1,Z_2):=\sum_{j=1}^m \frac{Z_1^{(j-1)n}Z_2^{m-j}c_1^{j-1} }{c_2^{jn}}.\]
\end{thm}

\begin{proof}
	Note that the curve $W_{a,b,c}$ is cut out by the polynomials:
	\begin{equation*}
		G_{(c,a\omega,1)}(X_1,X_3)=cX_1-a\omega X_3 \quad \text{and} \quad G_{ (b,a,r-2)}(X_1,X_2)=bX_1^{r-2}-a^{r-2}X_2.
	\end{equation*}
Using \eqref{eq:ProdElem}, we have 
		\begin{align*}
	&G_{(c,a\omega,1)}(X_1,X_3)\left(S_{(c,a\omega,1,r)}(X_1,X_3)+\frac{X_1X_2}{\omega^{r-1}}S_{(c,a\omega,1,1)}(X_1,X_3)\right)\\
	&+G_{ (b,a,r-2)}(X_1,X_2)\left( S_{(b,a,r-2,q)(X_1,X_2)} + \frac{cX_1^2}{a\omega^r}S_{(b,a,r-2,1)(X_1,X_2)} \right)\\
	&=\frac{c^r X_1^r}{a^r \omega^r}-X_3^r -\frac{X_1X_2X_3}{\omega^{r-1}}+\frac{b^qX_1^{q(r-2)}}{a^{q(r-2)}}-X_2^q + \frac{cbX_1^r}{a^{r-1}\omega^r} \mbox{ and }\\
	& \frac{c^r X_1^r}{a^r \omega^r}+\frac{cbX_1^r}{a^{r-1}\omega^r}=0\, \, \mbox{ and }\, \, 
	\frac{b^qX_1^{q(r-2)}}{a^{q(r-2)}}=-X_1^{q(r-2)}.
		\end{align*}
where the equalities in the last line follows from the hypothesis \eqref{eq:ass}.
	Using Corollary \ref{thm1} we conclude that the matrix factorization associated to $M_{a,b,c}$ is
	as given in the statement of the theorem. This proves the theorem.
\end{proof}
	
	\begin{rem}
	Note that:
	\begin{enumerate}
	 \item If we assume $q=r=3$, then $F$ is the cubic polynomial studied by Etingof and Ginzburg~\cite{ginz}.
	 \item If we impose the inequality $r < q(r-2)$, then  $F$ is a cusp singularity of type $T_{(r-2)q,q,r}$ (see~\cite[Theorem~7.10]{burb}).
	\end{enumerate}
	\end{rem}

\subsection{Non-isolated singularities} \label{sec:nonIsolated}
Our next application is to show how to generalize the construction of Baciu~\cite{baci}. Consider the homogeneous polynomial
\begin{equation*}
	F= X_1^{4}+X_1^{3}X_3-X_2^{4}X_3.
\end{equation*}
In this case the singular locus is the line
\begin{equation*}
	X_{\mathrm{sing}}=\{(0,0,z) \in \CC^3 \, | \, z\in \CC\}.
\end{equation*}
Let $(a,b,1) \in \mb{C}^3 \setminus \{0\}$ such that $F(a,b,1)=0$.  The $\CC^*$-curve given by the point $(a,b,1)$ is the zero locus of the ideal given by
\begin{equation*}
	f=X_1-X_3a \quad \text{and} \quad g=X_2-X_3b.
\end{equation*}
Let $h_1 =X_1^3+X_1^2X_2+aX_1^2X_3+aX_1X_2X_3+a^2X_1X_3^2+a^2X_2X_3^2+a^3X_3^3$ and 
\[ h_2 =X_2^2X_3+bX_2X_3^2+(a^3+b^2)X_3^3.\]
We then have the corresponding matrix factorizations of $F$:
\begin{equation*}
	M(a,b,c)=\begin{pmatrix}
		f & g\\-h_2& h_1 \end{pmatrix}.
\end{equation*}
Notice that $h_1$ and $h_2$ can be rewritten as:
\begin{align*}
	h_1 =&X_1(X_1^2+X_1X_2+aX_1X_3+aX_2X_3)+X_3(a^2X_1X_3+a^2X_2X_3+a^3X_3^2),\\
	h_2 =&X_3(X_2^2+bX_2X_3+(a^3+b^2)X_3^2).
\end{align*}
Therefore, the following matrices (also parameterized by the points $(a:b:1) \in X^*$) are matrix factorizations of $F$:
\begin{equation*}
	M(a,b,c;3)=\begin{pmatrix}
		0 & -f & g\\
		X & -X_2^2-bX_2X_3-(a^3+b^2)X_3^2 &-a^2X_1X_3-a^3X_3^2 \\
		X_3 & 0 &X_1^2+X_1X_2+aX_1X_3+aX_2X-aX_2X_3 \end{pmatrix}.
\end{equation*}

\section*{Acknowledgement}
We thank Prof. Javier F. de Bobadilla and Prof. Duco van Straten for 
	their interest in this problem and helpful comments. The first author is funded by 
	EPSRC grant number EP/T019379/1. The second author is funded by OTKA 126683 and Lend\"ulet 30001. The second author thanks CIRM, Luminy, for its hospitality and for providing a perfect work environment. He also thanks Prof. Javier F. de Bobadilla, the 2021 semester 2 Jean-Morlet Chair, for the invitation.

\end{document}